\documentclass[11pt,a4paper]{article}

\usepackage[T1]{fontenc}
\usepackage[utf8]{inputenc}
\usepackage{lmodern}
\usepackage{microtype}
\usepackage[a4paper,margin=29mm]{geometry}
\usepackage{amsmath,amssymb,amsthm,mathtools}
\usepackage{xcolor}
\usepackage[colorlinks=true,
  linkcolor=blue!55!black,
  citecolor=blue!55!black,
  urlcolor=blue!55!black]{hyperref}

\hypersetup{
  pdftitle={The N-Prime Graph Question is equivalent to the Prime Graph Question},
  pdfauthor={Brecht Verbeken},
  pdfsubject={Integral group rings, N-prime graph, Prime Graph Question}
}

\newtheorem{theorem}{Theorem}[section]
\newtheorem{proposition}[theorem]{Proposition}
\newtheorem{lemma}[theorem]{Lemma}
\newtheorem{corollary}[theorem]{Corollary}
\theoremstyle{remark}
\newtheorem{remark}[theorem]{Remark}

\newcommand{\VZG}{V(\mathbb ZG)}
\newcommand{\GN}{\Gamma_{\mathrm N}}
\newcommand{\GGK}{\Gamma_{\mathrm{GK}}}
\newcommand{\Aut}{\operatorname{Aut}}
\newcommand{\Sym}{\operatorname{Sym}}
\newcommand{\eps}{\varepsilon}

\title{The $N$-Prime Graph Question\\
is equivalent to the Prime Graph Question}
\author{Brecht Verbeken\thanks{Department of Business Technology and Operations,
Data Analytics Laboratory, Vrije Universiteit Brussel (VUB), Pleinlaan 2,
1050 Brussels, Belgium; and imec-SMIT, Vrije Universiteit Brussel,
Pleinlaan 9, 1050 Brussels, Belgium. Corresponding author:
\href{mailto:brecht.verbeken@vub.be}{\texttt{brecht.verbeken@vub.be}}.}}
\date{}

\begin{document}
\maketitle

\begin{abstract}
Let $G$ be a finite group and let $V(\mathbb ZG)$ be the group of
normalized units of its integral group ring. We prove that every $N$-prime
arc of $V(\mathbb ZG)$ either already occurs in $G$ or admits commuting
witnesses of distinct prime orders. Writing $A(\Delta)$ for the arc set
of a directed graph $\Delta$, $E(\Delta)$ for the edge set of an
undirected graph, and $\operatorname{Sym}(E)$ for the two orientations of
the edges in $E$, this is equivalent to the exact formula
\[
 A\bigl(\Gamma_{\mathrm N}(V(\mathbb ZG))\bigr)
 =
 A\bigl(\Gamma_{\mathrm N}(G)\bigr)
 \cup
 \operatorname{Sym}\!\bigl(
 E(\Gamma_{\mathrm{GK}}(V(\mathbb ZG)))
 \bigr).
\]
Consequently, the $N$-Prime Graph Question has an affirmative answer for
$G$ if and only if the Prime Graph Question does.
\end{abstract}



\section{Introduction}

Let $G$ be a finite group, let $\eps\colon\mathbb ZG\to\mathbb Z$ be the
augmentation map, and write
\[
 \VZG=\{u\in U(\mathbb ZG):\eps(u)=1\}.
\]
For a group $X$, let $\pi(X)$ be the set of primes that occur as orders of
torsion elements of $X$. The Gruenberg--Kegel graph $\GGK(X)$ is the simple
undirected graph with vertex set $\pi(X)$ in which distinct vertices $p$ and
$q$ are adjacent if $X$ contains an element of order $pq$. Kimmerle
\cite{Kimmerle2006} introduced the Prime Graph Question (PQ), which asks
whether $\GGK(\VZG)=\GGK(G)$ for every finite group $G$. Kimmerle and
Konovalov reduced PQ to almost simple groups \cite{KimmerleKonovalov}; see
also the survey \cite{MargolisDelRioSurvey}.

Pacifici, del R\'io and Vergani \cite{PRV} introduced the $N$-prime graph
$\GN(X)$ as the simple directed graph with vertex set $\pi(X)$ in which,
for distinct vertices $p$ and $q$, the arc $q\to p$ occurs if some subgroup
$P$ of $X$ of order $p$ has a normalizer containing an element of order
$q$. They formulated the $N$-Prime Graph Question (NPQ), asking whether
$\GN(\VZG)=\GN(G)$, as a strengthening of PQ. They also proved NPQ for
solvable groups and reduced it to almost simple groups
\cite[Theorem~A and Corollary~B]{PRV}. As an invariant of groups, the
$N$-prime graph is strictly finer than the Gruenberg--Kegel graph: its
double arrows determine the latter graph, whereas its single arrows
contain additional normalizer information \cite[Remark~3.1]{PRV}.

The surprising point is that this strictly finer invariant produces no
stronger equality question for integral group rings:
\[
 \GN(\VZG)=\GN(G)
 \quad\Longleftrightarrow\quad
 \GGK(\VZG)=\GGK(G).
\]
The reason is a dichotomy for $N$-prime arcs in $\VZG$. If the normalizing
element centralizes the subgroup of prime order, the two units commute and
give an element of order $pq$. If the action is nontrivial, a
partial-augmentation orbit argument transfers the corresponding $N$-prime
arc to $G$. This dichotomy gives not only the equivalence but also an exact
formula for the arc set. We shall also use
\[
 \pi(\VZG)=\pi(G).
\]
Indeed, if $p\in\pi(\VZG)$, then $p\mid\exp(G)$ by
\cite[Corollary~4.1]{CohnLivingstone}, and hence $p\in\pi(G)$. The reverse
inclusion follows from $G\leq\VZG$. Thus all graphs in the following
theorem have a common vertex set. For a directed graph $\Delta$, write
$A(\Delta)$ for its arc set, and for an undirected graph $\Delta$, write
$E(\Delta)$ for its edge set. For an undirected edge set $E$, put
\[
 \Sym(E)=\{(p,q),(q,p):\{p,q\}\in E\}.
\]

\begin{theorem}[Exact arc formula]\label{thm:main}
Let $G$ be a finite group. Then
\begin{equation}\label{eq:identity}
 A\bigl(\GN(\VZG)\bigr)
 =
 A\bigl(\GN(G)\bigr)
 \cup
 \Sym\!\bigl(E(\GGK(\VZG))\bigr).
\end{equation}
Consequently,
\begin{equation}\label{eq:equivalence}
 \GN(\VZG)=\GN(G)
 \quad\Longleftrightarrow\quad
 \GGK(\VZG)=\GGK(G).
\end{equation}
\end{theorem}

\begin{corollary}\label{cor:counterexample}
If an arc $q\to p$ occurs in $\GN(\VZG)$ but not in $\GN(G)$, then
$\VZG$ contains an element of order $pq$. In particular, $G$ is a
counterexample to the $N$-Prime Graph Question if and only if it is a
counterexample to the Prime Graph Question.
\end{corollary}

\begin{proof}
This follows immediately from Theorem~\ref{thm:main}.
\end{proof}

The equivalence has immediate structural consequences. The almost-simple
reduction for NPQ in \cite[Theorem~A]{PRV} follows formally from the
Kimmerle--Konovalov reduction for PQ \cite{KimmerleKonovalov}, and the
solvable case \cite[Corollary~B]{PRV} follows from the corresponding result
for PQ \cite{Kimmerle2006}. More generally, every affirmative result for PQ
on a class of finite groups automatically yields NPQ on the same class.
These graph-equality consequences do not subsume the stronger
subgroup-isomorphism results of Pacifici, del R\'io and Vergani, which
require a prescribed subgroup isomorphism type.

The closest overlap in the proof is \cite[Lemma~4.2]{PRV}. Their argument
also studies the $r$-power permutation of the conjugacy classes of elements
of order $p$: under hypotheses on characters and Galois transport, equality
of partial augmentations along the $q$-orbits leads, when there is no fixed
class, to an augmentation contradiction. Here the coprime-power identity
for partial augmentations, together with the unit conjugacy $u\sim u^r$,
gives orbitwise constancy directly, without the character-theoretic
hypotheses used there. The augmentation sum forces an $r$-fixed
class; such a class yields an element of order $q$ in the automizer
\[
 N_G(\langle a\rangle)/C_G(a),
\]
and hence forces $q\mid |N_G(\langle a\rangle)|$. Combining this with the
commuting case gives the exact arc formula.

\section{Graph and partial-augmentation preliminaries}

\begin{lemma}[Double-arrow principle]\label{lem:double}
For distinct primes $p,q$ and any group $X$,
\[
 \{p,q\}\in E(\GGK(X))
 \quad\Longleftrightarrow\quad
 (q,p),(p,q)\in A(\GN(X)).
\]
\end{lemma}

\begin{proof}
An element of order $pq$ has commuting $p$- and $q$-parts, which give both
arcs. Conversely, suppose that both arcs occur but that $X$ has no element
of order $pq$. A witness for $q\to p$ cannot centralize its subgroup of
order $p$, since commuting elements of coprime orders $p$ and $q$ have a
product of order $pq$. Hence an element of order $q$ acts nontrivially on
$C_p$, so $q$ divides $p-1$. Reversing $p$ and $q$ gives that $p$ divides
$q-1$, which is impossible.
\end{proof}

For $a=\sum_{g\in G}a_g g\in\mathbb ZG$ and a conjugacy class $C$ of $G$,
define the partial augmentation
\[
 \eps_C(a)=\sum_{g\in C}a_g.
\]
Let $\mathcal C_p(G)$ denote the set of conjugacy classes of elements of
order $p$. If $u\in\VZG$ has prime order $p$, the standard support theorem
\cite[Theorem~2.3]{Hertweck}, together with the Berman--Higman theorem
\cite[Proposition~(1.4)]{Sehgal}, gives
\begin{equation}\label{eq:support}
 \eps_C(u)=0\quad(C\notin\mathcal C_p(G)),
 \qquad
 \sum_{C\in\mathcal C_p(G)}\eps_C(u)=1.
\end{equation}

Conjugation by ring units preserves partial augmentations. Indeed, for
$w\in U(\mathbb ZG)$,
\[
 w^{-1}uw-u=[w^{-1}u,w],
\]
where $[a,b]=ab-ba$ denotes an additive commutator, and the coefficient sum
of an additive commutator on each conjugacy class is zero. Hence
\begin{equation}\label{eq:conjugation}
 \eps_C(w^{-1}uw)=\eps_C(u).
\end{equation}
For a conjugacy class $C$ and an integer $k$, write $C^k$ for the class
containing the $k$th powers of the elements of $C$.

\begin{lemma}\label{lem:power}
Let $u\in\VZG$ have order $p$, where $p$ is prime, and let $r$ be coprime
to $p$. Then
\[
 \eps_C(u^r)
 =
 \eps_{C^{r^{-1}}}(u)
 \qquad
 (C\in\mathcal C_p(G)),
\]
where $r^{-1}$ is taken modulo $p$.
\end{lemma}

\begin{proof}
Choose an integer $s$ such that $rs\equiv1\pmod p$. By
\cite[Lemma~3.2(a)]{PAP},
\[
 \eps_C(u^r)
 =
 \sum_{\substack{D\ \text{a conjugacy class of }G\\ D^r=C}}
 \eps_D(u).
\]
By \eqref{eq:support}, only classes in $\mathcal C_p(G)$ contribute.
Because $\gcd(r,p)=1$, the map $D\mapsto D^r$ is a permutation of
$\mathcal C_p(G)$ with inverse $D\mapsto D^s$. Hence the only contributing
class is $C^s$, and the asserted identity follows.
\end{proof}

\section{Transfer of nontrivial \texorpdfstring{$N$}{N}-prime arcs}

For a prime $p$ and an integer $r$ coprime to $p$, let
$\operatorname{ord}_p(r)$ denote the multiplicative order of $r$ modulo
$p$.

\begin{proposition}[Transfer of nontrivial arcs]\label{prop:transfer}
Let $p$ and $q$ be distinct primes. Suppose that $u,v\in\VZG$ satisfy
\[
 |u|=p,\qquad |v|=q,\qquad v^{-1}uv=u^r,
 \qquad \operatorname{ord}_p(r)=q.
\]
Then the map $C\mapsto C^r$ fixes a conjugacy class of elements of order
$p$ in $G$. Consequently, $q\to p$ is an arc of $\GN(G)$.
\end{proposition}

\begin{proof}
Let
\[
 \rho\colon\mathcal C_p(G)\longrightarrow\mathcal C_p(G),
 \qquad \rho(C)=C^r.
\]
Since $\gcd(r,p)=1$, the map $\rho$ is a permutation. Conjugation
invariance of partial augmentations and Lemma~\ref{lem:power} give
\[
 \eps_C(u)
 =
 \eps_C(v^{-1}uv)
 =
 \eps_C(u^r)
 =
 \eps_{\rho^{-1}(C)}(u)
 \qquad(C\in\mathcal C_p(G)).
\]
Thus the partial augmentations of $u$ are constant on the $\rho$-orbits.
Moreover, $r^q\equiv1\pmod p$ implies $\rho^q=\operatorname{id}$. Hence
every $\rho$-orbit has size dividing the prime $q$, and therefore has size
$1$ or $q$. By \eqref{eq:support},
\[
 1
 =
 \sum_{C\in\mathcal C_p(G)}\eps_C(u)
 \equiv
 \sum_{\substack{C\in\mathcal C_p(G)\\ C^r=C}}
 \eps_C(u)
 \pmod q.
\]
In particular, the power map has a fixed class.

Let $a^G$ be such a class. Since $a^r$ is conjugate to $a$, choose $x\in G$
with $x^{-1}ax=a^r$. Since $a^r$ generates $\langle a\rangle$, we have
$x\in N_G(\langle a\rangle)$. Under the standard conjugation homomorphism
\[
 N_G(\langle a\rangle)\longrightarrow
 \Aut(\langle a\rangle),
 \qquad y\longmapsto(t\longmapsto yty^{-1}),
\]
the element $x^{-1}$ induces $a\mapsto a^r$. Its image therefore has order
$\operatorname{ord}_p(r)=q$. Thus
\[
 q\mid\bigl|N_G(\langle a\rangle)/C_G(a)\bigr|,
\]
and in particular $q\mid |N_G(\langle a\rangle)|$. Cauchy's theorem now
gives an element of order $q$ in this normalizer. Thus $q\to p$ is an arc
of $\GN(G)$.
\end{proof}

\begin{remark}
The element of order $q$ obtained from Cauchy's theorem need not induce the
automorphism $a\mapsto a^r$; it may centralize $\langle a\rangle$. Thus
Proposition~\ref{prop:transfer} transfers the $N$-prime arc, but not a
prescribed subgroup $C_p\rtimes C_q$.
\end{remark}

\section{The exact arc formula}

\begin{proof}[Proof of Theorem~\ref{thm:main}]
The inclusion $G\leq\VZG$ gives
\[
 A(\GN(G))\subseteq A(\GN(\VZG)).
\]
If $\{p,q\}\in E(\GGK(\VZG))$, the $p$- and $q$-parts of a unit of order
$pq$ commute and give both arcs $q\to p$ and $p\to q$. Thus the
right-hand side of \eqref{eq:identity} is contained in its left-hand side.

Conversely, let $q\to p$ be an arc of $\GN(\VZG)$. Let
$P=\langle u\rangle$ be the witnessing subgroup of order $p$, and let $v$
be a witnessing element of order $q$. Write $v^{-1}uv=u^r$, where
\mbox{$\gcd(r,p)=1$}. If $r\equiv1\pmod p$, then $u$ and $v$ commute, and their
product has order $pq$. Hence $\{p,q\}$ is an edge of $\GGK(\VZG)$.
Otherwise,
$v^q=1$ gives $r^q\equiv1\pmod p$. The multiplicative order
$\operatorname{ord}_p(r)$ therefore divides the prime $q$ and is not $1$,
so it equals $q$. Proposition~\ref{prop:transfer} then shows that $q\to p$
already occurs in $\GN(G)$. This proves \eqref{eq:identity}.

Suppose first that PQ holds for $G$. Then
\[
 E(\GGK(\VZG))=E(\GGK(G)).
\]
Every edge of $\GGK(G)$ yields both $N$-prime arcs in $G$, and substitution
in \eqref{eq:identity} gives
\[
 A(\GN(\VZG))=A(\GN(G)).
\]
The common vertex set now gives $\GN(\VZG)=\GN(G)$, so NPQ holds for $G$.
Conversely, if NPQ holds for $G$, Lemma~\ref{lem:double} recovers equal
Gruenberg--Kegel graphs from the equal sets of double arrows. This proves
\eqref{eq:equivalence}.
\end{proof}

\section*{Acknowledgements}

The author is grateful to \'Angel del R\'io and Emanuele Pacifici for
carefully reading an earlier version of the manuscript, and for their
helpful feedback and generous encouragement.

\section*{Declaration on the use of automated tools}

OpenAI Codex and ChatGPT were used to assist with literature searches,
manuscript organization, language editing, and \LaTeX{} preparation.
The author verified the mathematical arguments and bibliographic
information and assumes full responsibility for the final text.

\end{document}